\documentclass[12pt]{article}

\usepackage{amsmath,amssymb,amsbsy,amsfonts,amsthm,latexsym, 
amsopn,amstext,amsxtra,euscript,amscd,texdraw}
\usepackage{tikz}

\numberwithin{equation}{section}

\newtheorem{theorem}{Theorem}
\newtheorem{lemma}{Lemma}

\newtheorem*{corollary}{Corollary}

\newcommand{\twolinesum}[2]{\sum_{\substack{{\scriptstyle #1}\\
{\scriptstyle #2}}}}

\numberwithin{theorem}{section}

\numberwithin{lemma}{section}

\newcommand{\threelinesum}[3]{\sum_{\substack{{\scriptstyle #1}\\
{\scriptstyle #2}\\ {\scriptstyle #3}}}}

\def\ca {{\cal A}}
\def\cb {\mathcal{B}}

\def\car {{\cal R}}

\def\cj {{\cal J}}

\def\cf {{\cal F}}
\def\cg {{\cal G}}
\def\ch {\mathcal{H}}

\def\ck {{\cal K}}
\def\ci {{\cal I}}
\def\cl {{\cal L}}
\def\cs {{\cal S}}

\def\cm {{\cal M}}

\def\beq {\begin{equation}}
\def\endq {\end{equation}}
\def\frakD{{\mathfrak D}}
\def\ba {\boldsymbol{\alpha}}
\def\bl {\boldsymbol{\lambda}}

\def\bw {{\bf w}}
\def\sq {{\sum_{\chi \bmod{q}}\kern-6pt}^*}

\begin{document}

\title{Representing an integer as the sum of a prime and the product of two small factors}

\author{ROGER BAKER and GLYN HARMAN}

\maketitle

\begin{abstract}
\noindent \emph{Abstract.}
Let $\epsilon > 0$. We show that every large integer $n$ may be written in the form
\[
n=p+ab\,,
\]
where $a,b \le n^{\frac12 - \delta}$ for a positive absolute constant $\delta$, and $ab \le n^{0.55 + \epsilon}$. This sharpens a result of Heath-Brown \cite{DRHB}. The improvement depends on a lower bound version of Bombieri's theorem in short intervals. In establishing such a result we shall need to ``intersect" two lower bound prime-detecting sieves, and we give a more general discussion on this point which may have further applications. 

\vskip6pt
\noindent
MSC (2020): 11N13 (Primary).

\end{abstract}

\section{Introduction}
Answering a question raised by M. Car, Heath-Brown \cite{DRHB} showed that a large integer $n$ can be written as
\beq\label{1.1}
n=p + ab
\endq
where $p$ is prime  and $a,b$ are positive integers,
\beq\label{1.2}
\max(a,b) \le p^{\frac12 - \delta}
\endq
(where $\delta$ denotes a positive absolute constant) with
\beq\label{1.3}
ab \le n^{\theta}
\endq
where $\theta < 1$. In \cite{DRHB} it is shown that $\theta$ may be assigned any value $> \frac34$. 

Recently L\"u and Ren \cite{Xlu} claimed a sharpening of the above result, requiring only $\theta > \frac35$ in \eqref{1.3}. Unfortunately, there is a substantial error in the last paragraph of \cite{Xlu}. 

In the present paper the following theorem is proved.

\begin{theorem}
Given $\epsilon > 0$, 
every large integer $n$ can be written in the form \eqref{1.1} with $a, b$ positive integers subject to \eqref{1.2} and
\[
ab \le n^{0.55 + \epsilon}.
\]
\end{theorem}

The key auxiliary result that we use  is given below as Theorem 1.2. We require some notation. Let $x$ be a large positive number. 
Write $\cl = \log x$ and $\rho(n)$ for the indicator function of the primes. Given an arithmetic function $f(n)$, let
\[
E_f(y,h;q,a) = \twolinesum{y-h<n\le y}{n \equiv a \bmod{q}} f(n) \, -\, \frac{h}{\phi(q)h_0} \sum_{y-h_0 < n \le y} f(n),
\]
where $\frac12 x < y \le x$ and $h_0 = x \exp\left(-3 \cl^{\frac13}  \right)$.
Constants implied by the $\ll$ notation will depend at most on two parameters $A$ and $\epsilon$ unless otherwise indicated. We write $B$ for a positive absolute constant, not the same at each occurrence. We comment on one consequence of this convention: when we put a hypothesis such as $\max(M,N) \ll x^{\theta} \cl^{-2A-B}$  in a result (see \eqref{2.4} below, for example), we mean the value $B$ is chosen to cancel out  all $\cl^B$ factors entering in the proof.  We write $n \asymp N$ to indicate $n \in [N/B,BN]$ and $n \sim N$ for $n \in (N/2,N]$. 

\begin{theorem}
Let $\theta = 0.55 + \epsilon$.  There is an arithmetic function $\lambda(n)$ with the following properties.

\smallskip
\noindent
\quad (i) We have
\[
\lambda(n) \le \rho(n) \quad (2 \le n \le x),
\]
and
\[
\text{\rm if} \ \lambda(n)  \ne 0 \  \  \text{\rm then}  \ \ p|n \Rightarrow p \ge x^{\frac{1}{10}}.
\]

\noindent
\quad (ii) If \ $\frac12 x \le y < x$, then
\[
\sum_{y-h_0 < n \le y} \lambda(n) \gg h_0 \cl^{-1}.
\]

\smallskip
\noindent
\quad (iii) Let $A>0$, $Q= x^{\theta - \frac12} \cl^{-C}$ where C=C(A).

\smallskip
\noindent
Then
\beq\label{1.4}
\sum_{q \le Q} \max_{(a,q)=1} \max_{h \le x^{\theta}} \max_{y \sim x}
\left|E_{\lambda}(y,h;q,a)  \right|
\ll x^{\theta} \cl^{-A}.
\endq
\end{theorem}

This result is closely related to work of Huxley and Iwaniec \cite{HuxIwan}, Perelli, Pintz and Salerno \cite{PPS,PPS2}, Timofeev \cite{Tim}, Kumchev \cite{Kumchev} and Harman, Watt and Wong \cite{HWW}, but not comparable to any result in these papers.

We require some further notation. The symbol $p$, with or without subscript, is reserved for primes. Let
\[
P(z) = \prod_{p<z} p \qquad \text{and} \quad \psi(n,z) = \begin{cases} 1 & \text{if} \ (n,P(z))=1 \\
0 & \text{otherwise.} \end{cases}
\]

As usual, we attack Theorem 1.2 by using Dirichlet characters. For a character $\chi$, let $\delta(\chi)=1$ if $\chi$ is principal and $\delta(\chi) = 0$ otherwise.  We note that principal characters are not primitive, except in the trivial case of modulus $1$. All the new work we carry out in this paper only involves primitive characters with modulus $q\ge 2$, so we will tacitly assume $\delta(\chi)=0$ in any results we quote from the literature. The notation
\[
\sq
\]
denotes a sum restricted to the primitive characters $\bmod{\,q}$. Let
\[
E_f(y,h;\chi) = \sum_{y-h<k\le y} f(k) \chi(k)
- \delta(\chi) \frac{h}{h_0} \sum_{y-h_0<k\le y} f(k) 
\]
for an arithmetic function $f$. The following lemma reduces the proof of Theorem 1.2 to establishing a mean value result for $E_f(y,h;\chi)$. 

\begin{lemma}
Suppose that $f$ satisfies
\beq\label{1.5}
\sum_{q \sim Q} \sq \max_{h \le x^{\theta}} \max_{y \sim x} \left| E_f(y,h;\chi)  \right| \ll Q x^{\theta} \cl^{-A}
\endq
for all $A>0$ and $Q \le Q_0 = Q(A)$.  Suppose further that $|f(k)| \le d(k)^B$ and that
\beq\label{1.6}
f(k) = 0 \ \text{\rm if} \ \ k \ \ \text{\rm has a prime factor} \ < x^{\epsilon}.
\endq
Then, for all $A>0$ and the value $Q_0$ as above, we have
\[
\sum_{q \le Q_0} \max_{(a,q)=1} \max_{h \le x^{\theta}} \max_{ y \sim x}
\left|E_{f}(y,h;q,a)  \right|
\ll x^{\theta} \cl^{2-A}. 
\]
\end{lemma}

\begin{proof}
This is established in \cite[\S 4.1]{Kumchev}.
\end{proof}

We already have in \cite{BHP} results of the type \eqref{1.5} suitable for use in Theorem 1.2 when $Q \le \cl^{A+B}$. We develop similar results for $Q > \cl^{A+B}$ in sections 2 and 3.  To make these results work in tandem, we develop the idea of intersecting two prime-detecting sieves in sections 4 and 5. As we do this we will prove a result of a more general nature than needed here which may have other applications. Then, in section 6, we prove Theorem 1.2 using this approach. In section 7 we prove Theorem 1.1 by introducing the result of Theorem 1.2 into the method of Heath-Brown \cite{DRHB}.  We remark that the idea of intersecting two prime-detecting sieves is particularly useful in the situation we have here, when a lot of hard work has already been done to get a strong result for one of the sieves (see \cite{BHP}).

We would like to thank Andreas Weingartner for computer calculations used in an earlier version of this paper.

\section{Mean-Values of Dirichlet Polynomials}
Throughout sections 2 and 3 we suppose that
\beq\label{2.1}
\cl^{A+B} < Q \le Q_0.
\endq
For  $N, M, K \ge 1$, a character $\chi(\cdot)$ and a complex variable $s$, we write
\[
N(s,\chi) = \sum_{n \asymp N} b_n n^{-s} \chi(n), \quad 
M(s,\chi) = \sum_{m \asymp M} a_m m^{-s} \chi(m),
\]
\[
K(s,\chi) = \sum_{k \asymp K}  k^{-s} \chi(k)
\]
for Dirichlet polynomials, where
\beq\label{2.3}
a_m \ll d(m)^B, \qquad b_n \ll d(n)^B.
\endq  
For a fixed $Q > \cl^{A+B}$ and $x^2 \ge T \ge 1$ we write, for $1 \le u < \infty$,
\[
||N||_u = \left(\sum_{q \sim Q} \sq \int_{(T-1)/2}^T \left|N\left(\tfrac12+it,\chi\right) \right|^{u} dt\right)^{1/u}.
\]
We also write
\[
||N||_{\infty} =  \sup_{(t,q,\chi)} \left|N\left(\tfrac12+it,\chi\right)\right|
\]
where the supremum is taken over all primitive characters $\chi \pmod q, q \sim Q$ and $\frac12(T-1) \le t \le T$.

It is convenient to write (see Lemma 3.2 below for the application)
\[
\Psi(T) = \min\left(x^{\theta - \frac12}, x^{\frac12} T^{-1}  \right).
\]
Here, and throughout the rest of the paper, $\theta = 0.55+\epsilon$.

\begin{lemma}
We have
\[
||N||_2^2 \ll \left(N + Q^2T \right)
 \sum_{n \asymp N} |b_n|^2 n^{-1}.
\]
\end{lemma}

\begin{proof}
See \cite[Theorem 7.1]{Mon71}.
\end{proof}

\begin{lemma}
Let
$M(s,\chi), N(s,\chi)$
be defined as above and
$MN \asymp x$.
Suppose that
\beq\label{2.4}
\max(M,N) \ll x^{\theta} \cl^{-2A-B}.
\endq
Then
\beq\label{2.5}
\Psi(T) ||MN||_1 \ll Q x^{\theta} \cl^{-A}.
\endq
\end{lemma}

\begin{proof}
By the Cauchy-Schwarz inequality, Lemma 2.1, and \eqref{2.3} we have
\beq\label{2.6}
\Psi(T) ||MN||_1 \ll \Psi(T) \left(Q^2T + M  \right)^{\frac12} \left(Q^2T + N  \right)^{\frac12} \cl^B.
\endq
For $0 \le \beta \le 1$, we observe that
\[
\max_{T \ge 1} \Psi(T) T^{\beta} = x^{\theta-\frac12} \left(x^{1-\theta}  \right)^{\beta}.
\]
Thus, in estimating the right hand side of \eqref{2.6} we may suppose that $T=x^{1-\theta}$. Now
\[
x^{1-\theta} (MN)^{\frac12} \cl^B \ll x^{\theta} \cl^B \ll Q x^{\theta} \cl^{-A}
\]
from \eqref{2.1};
\begin{align*}
x^{\theta-\frac12} Q T^{\frac12} \max(M,N)^{\frac12} \cl^B &\ll 
x^{\theta-\frac12} Q x^{\frac12 (1-\theta)} x^{\frac12 \theta} \cl^{-A-B} \cl^B \\
&=Q x^{\theta} \cl^{-A};
\end{align*}
and
\begin{align*}
x^{\theta-\frac12} Q^2 x^{1-\theta} \cl^B &\ll 
x^{\theta-\frac12} Q  x^{\theta-\frac12}   \cl^{-A-B}   x^{(1-\theta)}  \cl^B \\
&=Q x^{\theta} \cl^{-A}.
\end{align*}
This gives a satisfactory estimate for the right-hand side of \eqref{2.6} and proves \eqref{2.5}. 
\end{proof}

\begin{lemma}
Let $Q \le x, 1 \le T \le x, K \le \min(x,QT)$ and suppose
$K(s,\chi)$ is as above.
Then
\[
||K||_4^4 \ll Q^2T \cl^B + Q^4T^{-1} \cl^B.
\]
\end{lemma}

\begin{proof}
This follows from \cite[Lemma 9]{BHP} on summing over $q$.
\end{proof}

\begin{lemma}
Let $q \ge 2$ and $\chi$ be a primitive character $\pmod q$ with $K(s,\chi)$ as above. Writing $\tau=|t|+2$, we have
\[
K(\tfrac12 + it, \chi) \ll  (q \tau/K)^{\frac12} \log(q \tau).
\]
\end{lemma}

\begin{proof}
This follows from \cite[Theorem 1]{FGM} by partial summation.
\end{proof}

\begin{lemma}
Let $K(s,\chi), M(s, \chi)$ and  $N(s,\chi)$ be defined as above.  
Suppose that $KMN \asymp x, Q \le x^{\theta - \frac12} \cl^{-A-B}$ and
\beq\label{3.6}
M \ll x^{0.55}, \quad N \ll x^{0.275}, \quad MN \ll x^{0.775}.
\endq
Then
\[
\Psi(T) ||KMN||_1 \ll Q x^{\theta} \cl^{-A}.  
\]
\end{lemma}

\begin{proof}
For small $Q$ this would follow from \cite[Lemma 10]{BHP} .  We give a full proof of the result here for completeness. We consider two cases.

\smallskip
\noindent
\emph{Case 1.} $K \le QT$.  We apply H\"older's inequality, Lemmas 2.1 and 2.3 in conjunction with \eqref{2.3} to give
\beq\label{2.8}
\begin{split}
\Psi(T) ||&KMN||_1 \le \Psi(T) ||M||_2||N||_4||K||_4\\
&\ll
\Psi(T) \left(Q^2T+M \right)^{\frac12} \left(Q^2T+N^2 \right)^{\frac14} \left(Q^2T+Q^4T^{-1} \right)^{\frac14} \cl^B.
\end{split}
\endq
Write $S$ for the sum of all the terms obtained by multiplying out the last expression, except for $\Psi(T) (MN)^{\frac12}QT^{-\frac14}$. In estimating $S$ we may take $T=x^{1-\theta}$.  Now
\[
S \ll \cl^B x^{\theta - \frac12} \left(Q^2 x^{1-\theta} + Q x^{\frac12(1-\theta)} M^{\frac12} + Q^{\frac32} x^{\frac34(1-\theta)}N^{\frac12} + Q^{\frac12} x^{\frac14(1-\theta) }(MN)^{\frac12}\right).
\]  
Since $Q \le x^{\theta- \frac12} \cl^{-A-B}$, we have
\begin{gather*}
\cl^B x^{\theta-\frac12} Q^2 x^{1-\theta} \ll   Q x^{\theta} \cl^{-A},\\
\cl^B x^{\theta - \frac12} Q x^{\frac12(1-\theta)} M^{\frac12} \ll Q \cl^B x^{\theta - \frac12 \epsilon},\\
\cl^B x^{\theta - \frac12} Q^{\frac32} x^{\frac34(1-\theta)}N^{\frac12}
\ll Q \cl^B x^{\theta- \frac14 \epsilon},
\end{gather*}
and
\[
\cl^B x^{\theta - \frac12} Q^{\frac12} x^{\frac14(1-\theta)} (MN)^{\frac12}
\ll Q \cl^B x^{\theta- \frac14 \epsilon}.
\]
This gives a satisfactory estimate for $S$. For the term $\Psi(T) (MN)^{\frac12}QT^{-\frac14}$ we use $T\ge1$:
\begin{align*}
\Psi(T) (MN)^{\frac12}QT^{-\frac14} &\ll x^{\theta - \frac12} x^{\frac14(1+\theta)} Q\\
&\ll Q x^{\theta} \cl^{-A}
\end{align*}
since $\frac14 (1+\theta) < \frac12$.  This covers Case 1.

\smallskip
\noindent
\emph{Case 2.} $K > QT$. From Lemma 2.4,
\begin{align*}
\Psi(T) ||KMN||_1 &\le \Psi(T) ||M||_2 ||N||_2 ||K||_{\infty}\\
&\ll
\Psi(T) \left(Q^2T+M  \right)^{\frac12} \left(Q^2T+N  \right)^{\frac12} \cl^B.
\end{align*}
Since 
\[
\left(Q^2T+N\right)^{\frac12} \ll \left(Q^2 T + N^2\right)^{\frac14} \left(Q^2 T  \right)^{\frac14}
\]
the result follows as in Case 1. This completes the proof of Lemma 2.5.
\end{proof}

\section{Sieve Estimates}
We first state Perron's formula in a suitable manner.
\begin{lemma}
Let $b>0, T>1$, and write
\[
E(u) = \begin{cases} 0 &\text{if} \ 0 < u < 1\\
1 &\text{if} \ u>1.
\end{cases}
\]
Then
\beq\label{3.1}
\frac{1}{2 \pi i} \int_{b-iT}^{b+iT} \frac{u^s}{s} \, ds = E(u) + O\left(\frac{u^b}{T |\log u|}  \right).
\endq
\end{lemma}
\begin{proof}
This is \cite[Lemma A.1]{GH} in different notation. 
\end{proof}

We note that, by standard procedures using the above, we can derive results of the form \eqref{1.5} for certain $f$ using Lemma 2.2 or 2.5. We do this explicitly for Lemma 2.2 giving the following result. 

\begin{lemma}
Suppose $x^{0.45} \le M \le x^{0.55}$ and $MN \asymp x$.  Write
\[
f(k) = \twolinesum{mn = k}{m \asymp M, n \asymp N} a_m b_n
\]
with $a_m, b_n$ subject to \eqref{2.3}. Then \eqref{1.5} holds for $f$.
\end{lemma}

\begin{proof} We first reduce the ``edge effects" by noting that we can restrict $y$ and $y-h$ to values $n + \frac12, n \in \mathbb{N}$ with a maximum error
\[
\sum_{q \sim Q} \sq  \max_{n \ll x} \tau(n)^B \ll Q^2 x^{\epsilon},
\]
which is certainly a negligible error.
It then follows from Lemma 3.1  that
\begin{align*}
E_f(y,h;\chi) &= \sum_{y-h<k\le y} f(k) \chi(k)\\
& =  \frac{1}{2 \pi i} \int_{\frac12-iT}^{\frac12+iT} M(s,\chi)N(s, \chi) \frac{y^s - (y-h)^s}{s} ds + 
O\left(V_{y,h}  \right),
\end{align*}
where
\[
V_{y,h} = \sum_{k \asymp x} \tau(k)^B \frac{1}{T \min(|\log(y/k)|, |\log((y-h)/k)|)}\,.
\]
A simple calculation gives
\[
\max_{\frac12 x \le y \le x} \, \max_{h \le x^{\theta}} V_{y,h} \ll \frac{x^{1 + \frac12 \epsilon}}{T}\,.
\]
If we take $T=x$ the error is again negligible from the above.

Next we note that
\[
 \frac{y^s - (y-h)^s}{s} = \int_{y-h}^y u^{s-1} du \ll hx^{-\frac12},
\]
and
\[
\frac{y^s - (y-h)^s}{s} \ll \frac{x^{\frac12}}{|s|}.
\]
It follows that, for $t \sim T$, we have
\[
\frac{y^s - (y-h)^s}{s} \ll \Psi(T).
\]
This demonstrates the significance of our choice for $\Psi(T)$.
It therefore remains to prove that, with $s=\frac12 + it$,
\[
\sum_{q \sim Q} \sq \max_{1 \le T \le x}  \Psi(T) \cl \int_{\frac12(T-1)}^{T} |M(s,\chi) N(s, \chi)| dt  \ll Q x^{\theta} \cl^{-A}.
\]
This follows immediately from Lemma 2.2 and completes the proof.
\end{proof}

The result we need to cover the case $Q> \cl^{A+B}$ is as follows. The result essentially follows for $Q \le \cl^{A+B}$ by the work done in \cite{BHP}. The result for $q=1$ goes back in principle at least as far as \cite{BH} (see Lemma 4 and following there). The proof of the result we need is actually simpler in principle since, in the case $Q> \cl^{A+B}$, we use Lemma 2.2 which is a straightforward Type II estimate that does not require one of
\[
||M||_{\infty} \ll M^{\frac12} \cl^{-A} \quad \text{or} \  ||N||_{\infty} \ll  N^{\frac12} \cl^{-A} 
\]
to hold.

\begin{lemma}
Suppose $M, N$ satisfy \eqref{3.6}.   Write
\[
f(k) = \twolinesum{mn \ell = k}{m \asymp M, n \asymp N} a_m b_n \psi(\ell, x^{0.1})
\]
with $a_m, b_n$ subject to \eqref{2.3}. Then \eqref{1.5} holds for $f$.
\end{lemma}

\begin{proof}
We could say: use the results of Lemmas 2.2 and 2.5 in the method of proof of \cite[\S 4]{BHP}, also given in \cite[\S 10.5]{GH}. The value $x^{0.1}$ comes from the ``width" of the Type sum in Lemma 2.2 (that is $x^{0.45}$ to $x^{055}$), and Lemma 2.5 is exactly analogous to the corresponding results in \cite{BH, BHP,GH}. By ``analogous", we mean here that our results are, as expected by the nature of the problem, a factor $Q$ larger than those in the previously cited works. However, as noted above, we are in a simpler situation (the preliminary proof of the result with $x^{0.1}$  replaced by $\exp\left( \cl^{0.9}\right)$ is unnecessary). We could therefore prove the result using a slight modification to \cite[Theorem 3.1]{GH}.
\end{proof}

By combining the two above lemmas it is possible to construct a lower bound prime detecting sieve using the method first delineated in \cite{h83} and developed in \cite{h96}, with many developments chronicled in \cite {GH}. This will be done explicitly in the following sections once we have outlined the general method and how to ``intersect" one sieve with another. The crucial additional result needed to construct the sieve is Buchstab's identity, which in the form we require, states that, for any $2 \le w < z$, we have
\[
\psi(k,z) = \psi(k,w) - \twolinesum{k=p \ell}{w \le p < z} \psi(\ell,p)\,. 
\]

\section{Some Further Notation}
Given an arithmetic function $f(n)$ we write $\bold{f}$ to mean the sequence $f(k), k \in [\tfrac12 x,x)\cap \mathbb{N}$.  So, for example, $\bl$ means the sequence $\lambda(k), \frac12 x \le k < x$. 
We recall the definition of a \emph{sublinear functional} $\cm$ defined over the set of all real sequences. That is, given sequences 
$\bold{w}, \bold{v}$,
we have
\[
\cm(c \bold{w}) = |c| \cm(\bold{w}) \ \ (c \in \mathbb{R}), \qquad
\cm(\bold{w} + \bold{v}) \le \cm(\bold{w}) + \cm(\bold{v})\,.
\]
We then define a sequence $\bf w$ to be \emph{$\cm$-regular} if it satisfies
\[
\cm(\bold{w}) \ll 1.
\]
Here we assume the parameter $x \rightarrow \infty$ and the implied constant is independent of $x$.
For a set $\Phi$ of sublinear functionals, a sequence $\bold{w}$ is said to be \emph{$\Phi$-regular} if $\bold{w}$ is $\cm$-regular for each $\cm \in \Phi$. 
 
\noindent
{\bf Example 1.}  Let $A > 0$; let $Q_0(A) = x^{\theta-\frac12} \cl^{-A-B}, \eta =\eta(Q)= x^{-\theta} Q^{-1} \cl^A$ and
\[
\cm_Q(\bw) = \eta \sum_{q \sim Q} \sq \max_{h \le x^{\theta}} \, \max_{y \sim x} \left|E_w(y, h, \chi)  \right|\,.
\]
We put
\[
\Phi_1=\left\{\cm_Q: \cl^{A+B} \le Q \le Q_0  \right\}\,, \quad \Phi_2=\left\{\cm_Q: 1 \le   Q \le  \cl^{A+B}  \right\}\,.
\]
Then (1.4) of Theorem 2 is the requirement that $\bl$ be both $\Phi_1$-regular and $\Phi_2$-regular.

In the following we write $\cb$ for the set of integers in the interval $(y-h_0,y)$ for some $y \sim x$,
\vskip5pt
\noindent
{\bf Example 2.} Let $\ca \subset \cb$ and let $\sigma$ be an approximation to $|\ca| |\cb|^{-1}$. We wish to show that $\ca$ contains primes. Let
\[
\cm(\bw) = \frac{\cl^2}{\sigma |\cb|} \left|\sum_{k \in \ca} w(k) - \sigma \sum_{k \in \cb} w(k)  \right|\,.
\]
Suppose an $\cm$-regular sequence $\bl$ satisfies $\lambda(k) \le \rho(k)$ and
\beq\label{4.1}
\sum_{k \in \cb} \lambda(k) \gg \sum_{k \in \cb} \rho(k)\,.
\endq
Then we have
\[
\left| \sum_{k \in \ca} \lambda(k) - \sigma \sum_{k \in \cb} \lambda(k)\,\right| \ll \frac{\sigma |\cb|}{\cl^2},
\]
and, from the prime number theorem for the slightly shortened interval \\ $(y-h_0,y)$,
\[
\sigma \sum_{k \in \cb} \lambda(k) \gg \frac{ \sigma |\cb|}{ \cl}\,.
\]
Consequently
\[
\sum_{k \in \ca} \rho(k) \ge \sum_{k \in \ca} \lambda(k) \gg \sigma |\cb| \cl^{-1}\,,
\]
and so $\ca$ contains primes.

By the sublinearity definition, we can thus establish lower bounds for the number of primes in sets by constructing $\bl$ as above using a linear combination of a bounded number of regular sequences.  In particular, to prove Theorem 1.2, we need to show that the $\bl$ we construct is a combination of sequences that are regular for both $\Phi_1$ and $\Phi_2$. We now briefly sketch the process developed in \cite{ h83, h96, GH} for constructing a lower bound sieve $\bl$. Let us only consider $k > \frac12 x$.  Then, for any value $z_0$,
\[
\rho(k) = \psi(k,x^{\frac12}) = \psi(k,z_0) - \twolinesum{z_0 \le p < x^{\frac12}}{k=p\ell} \psi(\ell,p) \,,
\]
using Buchstab's identity. We choose $z_0$ to be the largest value with $\psi(k,z_0)$ being regular. 
Now we may apply Buchstab's identity a bounded number of times -- indeed we want an \emph{even} (say $K$) number of times -- to give
\[
\rho(k) = \psi(k,z_0) - f_1(k) - f_3(k) - \ldots - f_{K-1}(k) + g_2(k) + \ldots +g_K(k)
\]
for certain functions $f_j,g_j$ defined as sums of the $\psi$ function over $j+1$ variables which multiply together to give $k$. By their definition, these functions are non-negative. Now, suppose each $f_j(k)$ is regular, and we split $g_j$ into a regular term and a (possibly non-existent) term which is not regular: $g_j(k) = r_j(k) + s_j(k)$, say. Then the function
\beq\label{lamdef}
\lambda(k) = \psi(k,z_0) - f_1(k) - f_3(k) - \ldots - f_{K-1}(k) + r_2(k) + \ldots + r_K(k)
\endq
is certainly regular by the sublinearity definition.  Since we have only discarded non-negative terms, we have $\lambda(k) \le \rho(k)$.  The challenge that remains is to carry out the above process in such a way that \eqref{4.1} holds. When successful, we shall call $\bl$ a \emph{non-trivial lower bound prime-detecting sieve} for whatever $\cm$ or family of $\cm$ is under consideration.   Clearly our ultimate challenge is to find such a $\bl$ for Example 1, regular for both families $\Phi_1$ and $\Phi_2$ defined there.

Given a sublinear functional $\cm$ (or family $\Phi$ of such), we call the process leading to \eqref{lamdef} a \emph{permissible decomposition for
$\cm$ (respectively, $\Phi$)}. We write $\frakD$ with or without a subscript for a permissible decomposition.  Now, if \eqref{4.1} holds with $\lambda(k) \le \rho(k)$, then clearly
\[
\sum_{k \in \cb} \lambda(k) =(1 - d) \sum_{k \in \cb} \rho(k)\,
\]
for some $d \in (0,1)$.  By what we have written so far, $d$ might vary somewhat with $x$, but we shall show later that in practice $d = \delta +o(1)$ for a fixed $\delta \in (0,1)$. We write this as $\delta(\frakD)$ to show that it depends essentially on the decomposition process (equivalently we could have defined it as $\delta(\bl)$).

\section{Intersecting Two Prime-Detecting Sieves}

When constructing a prime detecting sieve as outlined in \cite[Chapters 3, 5--11]{GH}, one first garners the necessary arithmetical information (what is commonly know as Type I \& II information, although there are sub-varieties of these types that appear as well), and then one applies Buchstab's identity to decompose a sum over primes in a given set, $\ca$, into multiple sums. Alternatively, and equivalently, as in the previous section we use Buchstab's identity applied to $\rho(k)$ summed over the set of interest, which often means we weight $\rho(k)$ with auxiliary functions like $k^{-s}, \chi(k)$ or $e(\alpha \ell k)$. 

Suppose $K$ is an even integer, and $\zeta>0$  (but not arbitrarily small)  is given. Write $\xi = x^{\zeta}$.
 For $j \le K < \zeta^{-1}$ write
\[
E_j=\{\ba_j: \alpha_i \ge \zeta \ (1 \le i \le j), \alpha_1+\ldots+\alpha_j \le 1\}.
\]
We write $\mathbf{p}_j=(p_1, \ldots,p_j)$ where
$p_i=x^{\alpha_i}, 1 \le i \le j, \ba_j \in E_j$, and
\[
\Pi_j = \prod_{1 \le i \le j} p_j\,.
\]
Given a Jordan measurable subset $\cf$ of $E_j$ we may thus write 
\[
\sum_{\boldsymbol{\alpha}_j \in \cf}
\]
to indicate the summation ranges for $\mathbf{p}_j$. In practice, $\cf$ will always be the interior of a polyhedron (or union of finitely many polyhedrons).
When we apply Buchstab's identity in a straightforward way we will have summation ranges satisfying
\[
p_j < p_{j-1}< \ldots < p_1, \qquad {p_r}^2 (p_1\ldots p_{r-1}) < x, \ \text{for} \ \ 2 \le r \le j.
\]
We write $\ch_j$ for the set of all possible values of $\boldsymbol{\alpha}_j$ in this case.

When we sketched out a decomposition in the previous section we assumed we would always decompose the ``inner variable" counted with the $\psi$ notation using Buchstab's identity. This is not always the most efficient course of action. Instead one can apply what is known as a r\^ole-reversal. We describe the simplest case of this technique (see also \cite[\S3.5]{GH}).  Suppose $\cf$ is a Jordan measurable region of $\ch_2$ and suppose we have a term to consider of the form
\[
w(k) = \twolinesum{\ba_2 \in \cf}{k=\Pi_2 \ell} \psi(\ell,p_2)\,.
\]
Instead of applying Buchstab's identity to $\psi(\ell,p_2)$, we rewrite $\bw$ as
\beq\label{rr}
w(k) = \twolinesum{\ba_2 \in \cf}{k=n_1 p_2 \ell} \psi(\ell,p_2) \psi\left(n_1, \left(\frac{x}{\ell p_2}  \right)^{\frac12}  \right)
\endq
where $n_1 = x^{\alpha_1}$, and apply Buchstab's identity to
\[
\psi\left(n_1, \left(\frac{x}{\ell p_2}  \right)^{\frac12}  \right)\,.
\]
Clearly one could alternatively have applied this process to decompose $p_2$ as well or instead of $p_1$.  The possibilities become rather complicated at this point (consider the working in \cite[pp.32-41]{BHP}). 
We will first set up the sieve machinery in the case where there are no r\^ole-reversals.  We shall need to take these into account later to prove our main theorems, but the general result we prove initially is not valid when r\^ole-reversals are used. For the purposes of clarity it may help the reader to see the simplest case first.  

Our first task is to describe the decomposition process in a way that will help develop the method for intersecting two sieves.  Let $\Phi$ be a set of sublinear functionals, and our desire is to construct $\bl$ which is $\Phi$-regular. 
Now, given any Jordan-measurable set $R_j \subset E_j$  we write $R_j^*$ for the set of all
$\ba_t \in E_t, t \ge j$ whose coordinates can be partitioned into $j$ sets (none of which is empty) such that if we form the vector $\ba_j$  by summing the $\alpha_i$ in the $j$  sets, then $\ba_j \in R_j$. Suppose there are sets $R_j$  and functions $z(\ba_j) \ge \xi$ such that if we write
\[
\car = \bigcup_{j} R_j^*
\]
then the sequence
\[
v_k = \twolinesum{k=\Pi_{t} \ell}{\ba_t \in V_t\cap \car} \psi(\ell,z'(\ba_t))
\]
is $\Phi$-regular for every Jordan-measurable subset  $V_t \subset E_t$ and all functions $z'(\ba_t) \le z(\ba_t)$. If $U_j$ is a Jordan-measurable subset of  $ E_j\cap\car$ then we say $U_j$ is  a \emph{Type I $\Phi$-regular domain}.
Let $U_j$ be such a Type I $\Phi$-regular domain with associated $z(\ba_j)$. 
Write
\[
w_k = \twolinesum{k=\Pi_{j} \ell p}{\ba_j \in U_j} \sum_{z'(\ba_j) \le p < z(\ba_j)}\psi(\ell,p) \,.
\]
Then, by Buchstab's identity, for any $z'(\ba_j)< z(\ba_j)$ we have
\[
w_k
=
\twolinesum{k=\Pi_{j} \ell}{\ba_j \in U_j} \psi(\ell,z(\ba_j))
-
\twolinesum{k=\Pi_{j} \ell}{\ba_j \in U_j} \psi(\ell,z'(\ba_j))\,,
\]
and so $\bw$ is $\Phi$-regular. For future reference we call this \emph{Property $\car$}.

 Results such as these usually follow from what the second-named author has called \emph{The Fundamental Theorem}: see \cite[Theorems 3.1, 5.2]{GH}. By the method of proof, if we get regularity with values $z(\ba_t)$, we also get regularity for any $z'(\ba_t) \in [2,z(\ba_t)]$.  In our present context, for $\Phi_1$-regularity, Lemma 3.3  gives 
\[
R_2 = \{\ba_2 \in E_2: \alpha_1 \le 0.55, \ \alpha_2 \le 0.275, \ \alpha_1 + \alpha_2 \le 0.775\},
\]
with $z(\ba_j) \equiv x^{ \frac{1}{10}}$. 
We would thus take $\zeta={\frac{1}{10}}$ in this case. We write, in general, $z(\ba_j) = x^{\zeta(\ba_j)}$, and note that $\zeta(\ba_j)$ will be piece-wise linear in each coordinate by the methods used (for example, see \cite[p.153, Diagram 7.3]{GH}) and $\zeta$ would be the infimum of all $\zeta(\ba_j)$. 

Now, if we use r\^ole-reversals, we might have in our sums instead of $p_i$  a ``sieved-variable" $n_i$ all of whose prime factors exceed $g(\ba_t)= x^{\gamma(\ba_t)}$ where $\gamma(\cdot)$ is piecewise linear in each variable. That is we are counting $n_i$ with a weight $\psi(n_i,g(\ba_t))$.  We can do this by breaking up the sum into multiple sums over appropriate domains. 
We remark that we can replace as many prime variables $p_i$ as we wish by this process.  We illustrate this by considering 
\eqref{rr}.  The variable $\ell$ there has all its prime factors exceeding $p_2$ so, for some integer $t < \zeta^{-1}$,
\[
w(k) =\sum_{3\le h\le t} \twolinesum{\ba_2 \in \cf}{k=n_1 p_2 p_3 \cdots p_h} \psi\left(n_1, \left(\frac{x}{p_2 p_3 \cdots p_h}  \right)^{\frac12}  \right)
\]
Here the variables $p_j, j \ge 3$ satisfy $p_h > p_{h-1} > \ldots >p_3 > p_2$. We can rewrite $w(k)$ then as
\[
w(k) =\sum_{3\le h\le t} \twolinesum{\ba_h \in \cf_h}{k=n p_2 p_3 \cdots p_h} \psi\left(n, \left(\frac{x}{p_2 p_3 \cdots p_h}  \right)^{\frac12}  \right)
\]
for certain regions $\cf_h \in E_h$. For example,
\[
\cf_2 = \{ \ba_2: (1-\alpha_1-\alpha_2, \alpha_1) \in \cf, \alpha_2>\alpha_1  \}\,,
\]
and
\[
\cf_3 = \{ \ba_3: (1-\alpha_1-\alpha_2-\alpha_3, \alpha_1) \in \cf, \alpha_3 > \alpha_2>\alpha_1  \}\,.
\]

Similarly, but with one crucial difference, given any Jordan-measurable set $S_j \subset E_j$  we define  $S_j^*$ to be the set of all
$\ba_t \in E_t, t \ge j$ whose coordinates can be partitioned into $j+1$ sets (at most one of which is empty) such that if we form the vector $\ba_j$  by summing the $\alpha_i$ in $j$ of the (non-empty) sets, then $\ba_j \in S_j$. Suppose there are sets $S_j$ such that if we write
\[
\cs = \bigcup_{j} S_j^*
\]
then the sequence
\[
v_k = \twolinesum{k=\Pi_{t}}{\ba_t \in V_t\cap \cs} 1
\]
is $\Phi$-regular for every Jordan-measurable subset  $V_t \subset E_t$. If $U_j$ is a Jordan-measurable subset of  $ E_j\cap\cs$ then we say $U_j$ is  a \emph{Type II $\Phi$-regular domain}.  This corresponds to Type II (and possible more complicated variants) information in the usual description of this sieve method. For $\Phi_1$ we  have just the one set $S_1=[0.45,0.55]$ by Lemma 3.2. So, in this case, $\cs$ is comprised of all regions in $j$ dimensions, $1 \le j \le \zeta^{-1}$, contained in $E_j$ where a sum of a subset of the coordinates lies in $S_1$. This is just the union of a finite number of polyhedra. It follows from our definition of $\cs$ that the sequence
\[
w_k = \twolinesum{k=n_1 \cdots n_t}{\ba_t \in V_t \cap \cs} c_1(n_1)\cdots c_t(n_t)
\]
is $\Phi$-regular for any Jordan-measurable subset $V_t \subset E_t$, where each $c_j(n)$ is either $\rho(n)$ or $\psi(n,g(\ba_t))$ for some function $g(\ba_t)= x^{\gamma(\ba_t)}$ where $\gamma(\cdot)$ is piecewise linear in each variable. Let us illustrate this with the simplest case.
Suppose we investigate
\[
\twolinesum{\ba_j \in V_j \cap \cs}{k=\Pi_j\ell} \psi(\ell,g(\ba_j)).
\] 
This can be written as the sum of $< \zeta^{-1}$ sequences corresponding to $k=\Pi_j p_{j+1}\cdots p_{t}$.
Each sequence is $\Phi$-regular from the definition of $\cs$, because $\ba_t$ belongs to a Jordan-measurable region (guaranteed by $\ba_j \in V_j$ and the subsequent coordinates only have piecewise linear restrictions).

Consider now how the decomposition is framed to construct a lower bound sieve.  We always start with
\beq\label{start}
\rho(k) = \psi(k,z_0) - \twolinesum{z_0 \le p_1 < x^{\frac12}}{k=p_1 \ell} \psi(\ell,z(\alpha_1))
+\twolinesum{z(\alpha_1)\le p_2 < p_1<x^{\frac12}}{k=\Pi_2 \ell} \psi(\ell, p_2).
\endq
The $z$ values are usually chosen to be the largest  such that the first two sequences on the right hand side above are $\Phi$-regular, and, writing $z_0 = x^{\zeta_0}$, we assume $[\zeta_0,\tfrac12]$ is a Type I $\Phi$-regular domain. As above, we let $\xi$ be the minimum taken by any $z$ value. We write the second term on the right of \eqref{start} as $f_1(k)$ to conform with \eqref{lamdef}.
We note that the final sum, which for future reference we shall call $\Sigma_2(k)$, has no terms with $p_2^2p_1>x$ (this corresponds to the condition $\ba_2 \in \ch_2$). Let $\cf \subset \ch_2$ correspond to the values of $\ba_2$  in the final sum above.  Let us write $\cf_1 = \cf \cap \cs$. So
\[
v_2(k) = \twolinesum{\ba_2 \in \cf_1}{k=\Pi_2 \ell} \psi(\ell, p_2)
\]
is $\Phi$-regular. Now let $\cg = \cf \setminus \cf_1$, and write $\cf_3$ for the subset of $\cg$ for which two further applications of Buchstab's identity are permissible, that is
\[
\cf_3=\left\{\ba_2 \in \cg \cap\car: \{\ba_3 \in \ch_3: \zeta(\ba_2) \le \alpha_3 < \alpha_2\}\subset \car \right\} \,.
\]
We have thus written
\[
 \twolinesum{\ba_2 \in \cf_3}{k=\Pi_2 \ell} \psi(\ell, p_2)
=
 \twolinesum{\ba_2 \in \cf_3}{k=\Pi_2 \ell} \psi(\ell, z(\ba_2)) - \threelinesum{\ba_2 \in \cf_3}{\zeta(\ba_2) \le \alpha_3 < \alpha_2}{k=\Pi_3 \ell} \psi(\ell, z(\ba_3)) + \Sigma_4(k)\,,
\]
say, where the first two sequences on the right hand side above are  $\Phi$-regular and we call the second of these $f_3(k)$.

Finally, we write $\cf_2 = \cf \setminus \left( \cf_1 \cup \cf_3  \right)$. We then have the term
\[
s_2'(k) =  \twolinesum{\ba_2 \in \cf_2}{k=\Pi_2 \ell} \psi(\ell, p_2)
\]
which at first sight we might discard as a positive non-regular term.  However, it is possible to rewrite $s_2(k)$ as the combination of sequences written as sums over two, three or more primes:
\[
s_2^2(k) + s_2^3(k)+s_2^4(k) + \ldots = 
\twolinesum{\ba_2 \in \cf_2}{k=\Pi_2 p_3} 1 + \twolinesum{\ba_3 \in \cj_3} {k=\Pi_3p_4} 1
+ \twolinesum{\ba_4 \in \cj_4} {k=\Pi_4p_5} 1 \ldots,
\] 
for certain sets $\cj_i$.   The sequence $s_2^2(k)$ must be discarded. But we can split each $\cj_i$ into $\cj_i\cap \cs$ and $\cj_i \setminus \cs$
(one or other may be empty, of course) to obtain another regular sequence and a sequence to be discarded.  We may thus write
\[
s_2'(k) = s_2(k) + \sigma_2(k)
\]
where $s_2(k)$ is discarded and $\sigma_2(k)$ is regular.
We then have 
\[
r_2(k) = v_2(k) + \sigma_2(k) + \twolinesum{\ba_2 \in \cf_3}{k=\Pi_2 \ell} \psi(\ell, z(\ba_2))
\]
in the notation of \eqref{lamdef}. We can present the regions $\cf_j$ on a diagram for $\Phi_1$-regularity as follows.

\vskip5pt

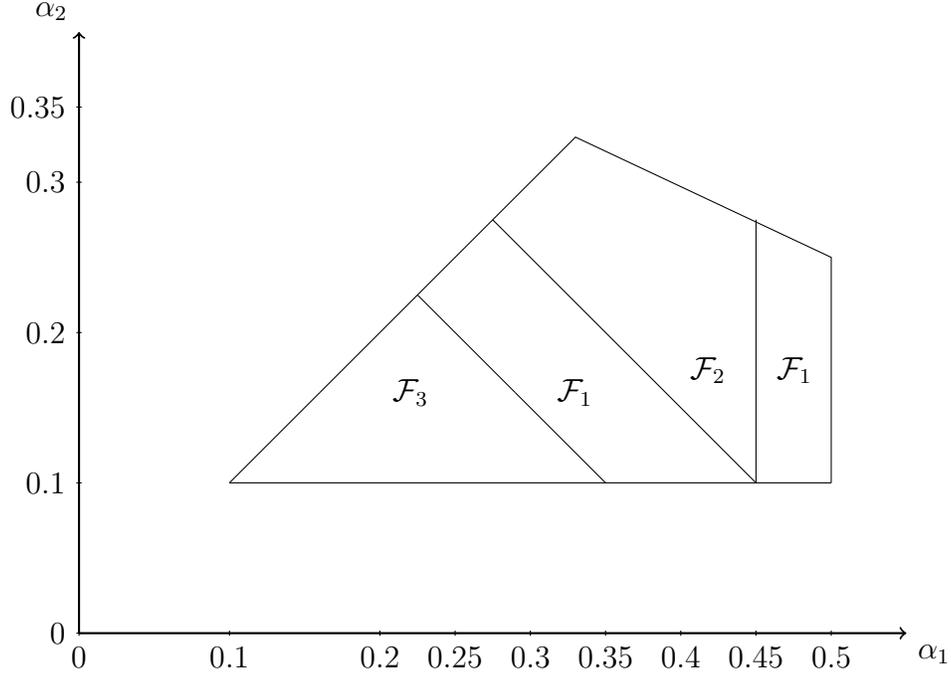
\begin{tikzpicture}
\draw[thick,->] (0,0) -- (11,0) node[anchor=north west] {$\alpha_1$};
\draw[thick,->] (0,0) -- (0,8) node[anchor=south east] {$\alpha_2$};
\draw(2,2)--(3.2,3.2) node [anchor=west] {\qquad $\cf_3$};
\draw(3.2,3.2)--(6.6,6.6)--(10,5)--(10,2);
\draw(9,2)--(9,3.5) node [anchor=west] {\kern-5pt $\cf_1$ \kern-50pt $\cf_2$};
\draw(9,3.5)--(9,5.5);
\draw(9,2)--(5.5,5.5);
\draw(7,2)-- (5.8,3.2) node [anchor=west] {\quad $\cf_1$};
\draw(5.8,3.2)--(4.5,4.5);;
\draw(2,2)--(10,2);
\foreach \x in {0,0.1,0.2,0.25,0.3,0.35,0.4,0.45,0.5}
   \draw (20*\x cm,1pt) -- (20*\x cm,-1pt) node[anchor=north] {$\x$};
\foreach \y in {0,0.1,0.2,0.3,0.35}
    \draw (1pt,20*\y cm) -- (-1pt,20*\y cm) node[anchor=east] {$\y$};
\end{tikzpicture}

\vskip10pt

Clearly we can analyze $\Sigma_4(k)$ as we did $\Sigma_2(k)$ and continue by induction to reach $\Sigma_K(k)$. For this last term we split into just two sums -- one corresponding to a region in $\cs$, and the remainder is discarded. We thus arrive at \eqref{lamdef} with $\bl$ the sum of $\Phi$-regular sequences, and so $\bl$  itself is $\Phi$-regular. We have thus defined our $\Phi$-permissible decomposition $\frakD$ leading to the lower bound sieve $\bl$ which is $\Phi$-regular.

Now suppose that $\cj_i \cap \cs = \emptyset$ for every $i$ at each stage of the decomposition (so $s_j'(k) = s_j(k)$).  Let $\cf_2(j)$ be the region corresponding to the discarded sum at each stage ($j$ even, $j \le K$). Let
\[
\delta_j = \int_{\cf_2(j)} \omega\left(\frac{1-\alpha_1 - \ldots - \alpha_j}{\alpha_j}\right) 
\frac{d\alpha_1 \ldots d\alpha_j}{\alpha_1 \ldots \alpha_{j-1}\alpha_j^2}\, ,
\]
and
\[
\delta(\frakD) = \twolinesum{j \ \text{even}}{j \le K} \delta_j.
\]
As explained in \cite[pp 16, 56--62]{GH},
\[
\sum_{k \in \cb} \lambda(k) =(1 - \delta(\frakD) + o(1)) \sum_{k \in \cb} \rho(k)\,.
\]
This value $\delta(\frakD)$ is independent of $x$ and so we can call it the \emph{deficit} of the lower bound sieve $\bl$. When 
 $\cj_i \cap \cs \ne \emptyset$ we can still define $\delta_j$ and hence $\delta(\frakD)$ independent of $x$, but the integrals concerned become much more complicated. For example, if we write
\[
\tau(\ba_j) = (\alpha_1 \cdots \alpha_j(1-\alpha_1-  \ldots - \alpha_j))^{-1},
\]
then
\[
\delta_2 = \int_{\ba_2 \in \cf_2(2)} \tau(\ba_2)  d\alpha_2 d \alpha_1
+
\sum_{j\ge 3} \int_{\ba_j \in \cj_j \setminus \cs} \tau(\ba_j)  d \alpha_j \ldots d\alpha_2 d \alpha_1 \,.
\]

We have now given enough notation and explanation to state and prove a general result on intersecting lower bound prime-detecting sieves.

\begin{theorem}
Let $\frakD_1, \frakD_2$ be two decompositions (constructed as above with no r\^ole-reversals) which are permissible for the families of sublinear functionals $\Xi_1,\Xi_2$ respectively. Then we can define a decomposition $\frakD_3 = \frakD_1 \cap \frakD_2$ which is permissible for $\Xi_1 \cup \Xi_2 $, and which satisfies
\beq\label{3}
\delta(\frakD_3) \le \delta(\frakD_1) + \delta(\frakD_2).
\endq
\end{theorem}

\begin{corollary}
Given the conditions of the theorem with $ \delta(\frakD_1) + \delta(\frakD_2) < 1$, then we can construct a non-trivial lower bound prime-detecting sieve $\bl$ which is both $\Xi_1$ and $\Xi_2$ regular.
\end{corollary}

\noindent
\emph{Proof of Theorem 5.1.}
For simplicity we suppose that $\cj_i \cap \cs = \emptyset$ for every $i$ at each stage for both decompositions. This does not change the principle involved in the proof  -- it just simplifies the details of describing the discarded sums at each stage.
Let $\car_1,\car_2, \cs_1, \cs_2$ correspond to $\frakD_1, \frakD_2$ respectively.
Our first task is to define $\frakD_3 = \frakD_1 \cap \frakD_2$. We do this  inductively as follows.
 We let $z_{0,j}$ and $z_j(\alpha_1)$ correspond to $\frakD_j$. We then put
\[
z_{0,3} = \min_{j=1,2} z_{0,j}, 
\qquad z_3(\alpha_1) = \min_{j=1,2} z_j(\alpha_1)\,.
\]
Now modify $\frakD_1, \frakD_2$ by replacing $z_{0,j}$ and $z_j(\alpha_1), j=1,2$ in their definition with the values for $j=3$. Of course, in some circumstances we may not have modified them at all, but in general the regions $\cf^{(j)}$ corresponding to our construction for each decomposition will have increased in size to a new $\cf^{(3)}$. Consider the region $\cf^{(3)}\setminus \cf^{(1)}$. For the sake of argument, suppose that $z_1(\alpha_1) > z_2(\alpha_1)$ for $\alpha_1 \in \ci$. This means that a new region
\[
\{(\alpha_1, \alpha_2): \alpha_1 \in \ci, \zeta_2(\alpha_1)< \alpha_2 < \zeta_1(\alpha_1)\}
\]
has appeared in the extended $\frakD_1$ decomposition. This increases the range of summation for the variable $\alpha_2$. However, by what we called \emph{Property  $\car$}, the sum over this new range is a $\Xi_1$-regular sequence. We get a similar result for $\cf^{(3)}\setminus \cf^{(2)}$. The crucial point is that increasing the summation domain does not disrupt the regularity property of the resulting sequence. So in the following we can redefine $\cf_1^{(1)}$ to be $\cf_1^{(1)}\cup (\cf^{(3)}\setminus \cf^{(1)})$, and similarly for $\cf_1^{(2)}$.

 In $\frakD_3$ we discard sums over the region $\cf_2^{(1)}\cup
\cf_2^{(2)}$. The sum over $\cf_1^{(1)}\cap\cf_1^{(2)}$ is clearly regular for both $\Xi_1$ and $\Xi_2$: call this term $u_2(k)$. Buchstab's identity can be applied twice more for $\cf_3^{(1)} \cap \cf_3^{(2)}$ leading to terms $v_2(k) - f_3(k) + \Sigma_4(k)$, say. Write $w_2(k) = u_2(k)+v_2(k)$. That only leaves the cases $\cf_1^{(i)} \cap \cf_3^{(j)}, \ i \ne j$ to discuss. Without loss of generality consider  $\cf_1^{(1)} \cap \cf_3^{(2)}$. In this case we work in $\frakD_3$ just as we would in $\frakD_2$ (that is, applying Buchstab's identity twice more) giving the same shape sums.  The important point to observe is that each of these sums is still regular for $\Xi_1$. To see this, simply note that in each of the new sums we still have $\alpha_1$ and $\alpha_2$ unchanged. Here we have observed that, keeping the first two coordinates fixed, $\ba_2 \in \cs_1 \Rightarrow \ba_3 \in \cs_1$ and $\ba_4 \in \cs_1$. This is a trivial consequence of our definition of a Type II $\Phi$-regular domain.
 We have thus started constructing our lower bound sieve $\bl$ which is regular for both $\Xi_1$ and $\Xi_2$. We have $\lambda(k) \le \rho(k)$ with
\[
\lambda(k) = \psi(k,z_{0,3}) - f_1(k) + w_2(k) - f_3(k)+ \Sigma_4(k),
\]
where the first $4$ terms on the right hand side above are  regular for both $\Xi_1$ and $\Xi_2$. As in our previous discussion we can treat $\Sigma_4$ similarly and continue to the case $K_3 = \max(K_1,K_2)$ by induction.  At each stage $\delta_j^{(3)} \le \delta_j^{(1)} + \delta_j^{(2)}$. Summing over $j$ then gives \eqref{3} as required.
\qed
\vskip5pt
Now, if r\^ole-reversals came into play, we can instantly see a problem might arise over a new type of region, even proceeding beyond $j=2$, where one  applies Buchstab twice more in $\frakD_1$ in a straightforward manner but uses a r\^ole reversal in $\frakD_2$, before applying Buchstab twice.   We are thus jumping straight from a two dimensional region to a six dimensional one.  Our philosophy would suggest we discard those parts of sums which are non-negative and not regular in the individual cases. This immediately gets complicated, because in mixing the two decompositions we then have to remove several sums, not all of which are non-negative.  We illustrate this as follows.  

Suppose at the $j=2$ stage we have the term
\[
\twolinesum{\ba_1 \in \cf}{k=p_1p_2\ell} \psi(\ell,p_2)
\]
which is regular for neither $\Phi_1$ nor $\Phi_2$.  Further suppose, that for $\Phi_1$ we can apply Buchstab twice more, and for $\Phi_2$ we rewrite it as per \eqref{rr}. This leads to the following two decompositions, where for simplicity (and since we are demonstrating that the result cannot be true in general) we shall assume that $z$ is fixed, for certain regions for the variables:
\beq\label{prob1}
\sum_{k=\Pi_2 \ell} \psi(\ell, z) - \sum_{k=\Pi_3 \ell} \psi(\ell, z) + \sum_{k=\Pi_4 \ell} \psi(\ell, p_4)\,,
\endq
\beq\label{prob2}
\sum_{k=np_2 \ell} \psi(\ell,p_2) \psi(n,z) - \sum_{k=np_2 p_5 \ell} \psi(\ell,p_2) \psi(n,z) + \sum_{k=np_2p_5p_6 \ell} \psi(\ell,p_2) \psi(n,p_6)\,.
\endq
We would expect to split the final sum in each of \eqref{prob1}, \eqref{prob2} into a regular term and a term which must be discarded.  Suppose that the discarded sum in \eqref{prob1} corresponds to $\ba_4 \in \ck$. We would then have to split the variable $\ell$ in each of the three terms in \eqref{prob2} into sums matching those in \eqref{prob1} and discard terms such as a sum,say $\Sigma'$, over (we state the most complicated)
$np_2p_3p_4p_5p_6 \ell$ with $np_5p_6 = x^{\alpha_1}$ and $\ba_4 \in \ck$.  Of course, at this point we are discarding some sums which could be negative (corresponding to the middle term in \eqref{prob2}).  However, on reversing Buchstab's identity it is clear that overall the discarded term is non-negative (indeed, it is $\delta_2^{(1)}$), and so the procedure is legitimate. It is a different matter, though, when we  come to discard that part of the final sum in \eqref{prob2} over a region not belonging to $\cs_2$, say we write this as ${\ba_4}' \in\ck'$ (${\ba_4}'$ corresponding to the variables $p_2, \ell, p_5,p_6$). We split into three cases with the variable $\ell$ unchanged, or decomposed as $\ell p_3$ or $\ell p_3 p_4$.  If we had not discarded $\Sigma'$ we could reverse Buchstab's identity and verify that the new discarded sum is overall non-negative (indeed, it would have been $\delta_2^{(2)}$). However, we cannot rule out that we are discarding negative terms and so our proof collapses.

We can put this argument another way. Let
\[
\sigma(k) = \threelinesum{k=np_2p_3p_4p_5p_6 \ell}{\ba_4 \in \ck}{{\ba_4}' \in \ck'} \psi(\ell,p_2) \psi(n,p_6)\,.
\]
Then, in general, we would have to assume that  $\sigma(k)$ is  regular for neither $\Phi_1$  nor $\Phi_2$ (though for some examples there might be a reason for it to be regular).  Now, if we were constructing a lower bound $\bl$ (regular for both $\Phi_1$  and $\Phi_2$) for $\rho(k) + \sigma(k)$, our method would work. But $\bl$ is no longer a lower bound prime-detecting sieve as it possibly gives a positive count for the numbers for which $\sigma(k)$ is non-zero.  We cannot retrieve the situation by subtracting off this term, because, since $\sigma(k)$  is not regular, $\lambda(k)-\sigma(k)$ cannot be regular.

\section{Application to Theorem 1.2}
Before proving Theorem 1.2 we first state a more general result which allows us to have r\^ole-reversals in one of the decompositions.
The trick, of course, is to ensure the problem that we described at the end of the last section cannot occur. To this end we establish the following result.

\begin{theorem}
Let $\frakD_1, \frakD_2$ be two decompositions which are permissible for the families of sublinear functionals $\Xi_1,\Xi_2$ respectively, and $\frakD_1$ involves no r\^ole-reversals.  Assume for $\frakD_1$ we apply Buchstab no more than 4 times, and at the $K=2$ stage there is a region $\cf_1$,  sums over which are $\Xi_1$-regular, a region $\cf_2$ which must be discarded, and a region $\cf_3$ where Buchstab can be applied twice more. Suppose further that $\cf_3=\cf_4 \cup \cf_5$ where sums over variables in $\cf_4$ are $\Xi_2$-regular (possibly using r\^ole-reversals) and, however the sums are treated over $\cf_5$ for $\frakD_2$, there are no r\^ole-reversals. Then \eqref{3} holds.
\end{theorem}

\begin{proof}
We only need to consider the two new cases introduced not covered by Theorem 5.1.  First of all, there may be a subset, say $\cg$ of $\cf_1$, where r\^ole-reversals are applied in the $\frakD_2$ decomposition.  This is easily dealt with, since, however the variables in $k=p_1p_2\ell$ are broken up, we can still combine them at any stage to give variables of the same size as $p_1, p_2, \ell$.  For $\cf_4$ no obstacle arises since we would only be discarding a four dimensional region in $\frakD_1$ and we have all the sums remaining $\Xi_2$-regular.  
\end{proof}

We can apply the above theorem directly to our problem with $\Xi_j=\Phi_j, j=1,2$. We have described the regions $\cf_j, 1 \le j \le3$ in the previous section. To cover $\frakD_2$ and the corresponding $\Phi_2$-regularity, consider the situation described in \cite{BHP} which gives
$\delta(\frakD_2) < 0.01$ when $\theta = 0.55 + \epsilon$. Here r\^ole-reversals are crucial. Looking at the proof, though, we notice that sums with $0.4 \le \alpha_1 +\alpha_2 \le 0.46$ are $\Phi_2$-regular \cite[Lemma 13]{BHP}.  Also, if $\alpha_1+\alpha_2 < 0.4$, no r\^ole-reversals are used (see \cite[p.37--38]{BHP}).  We can thus apply our theorem with $\delta(\frakD_2) < 0.01$.  Also, $\delta(\frakD_1)$ is the sum of the two integrals
\[
\int_{0.275}^{0.45} \int_{0.55-\alpha_1}^{\min(\alpha_1, (1-\alpha_1-\alpha_2)/2)} 
\omega\left(\frac{1-\alpha_1-\alpha_2}{\alpha_2}\right) \frac{d \alpha_2 d \alpha_1}{\alpha_2^2 \alpha_1}\,;
\] 
\[
\int_{\boldsymbol{\alpha}_4 \in \cg}
\omega\left(\frac{1-\alpha_1- \ldots\alpha_4}{\alpha_4}\right) \frac{d \alpha_4 \ldots d \alpha_1}{\alpha_4^2 \ldots \alpha_1}\,.
\]
Here $\cg$ is the region with $\boldsymbol{\alpha}_2 \in \cf_3, \boldsymbol{\alpha}_4 \in \ch_4, \alpha_4\ge 0.1$, and no sum of three or all four of
$\alpha_1,\ldots, \alpha_4$ is between $0.45$ and $0.55$.  

Calculations give the first integral as $<0.71$.  The second integral would have been $<0.16$ even before removing sums of variables between $0.45$ and $0.55$. With those removals the integral is $<0.02$.  So, we get a prime detecting lower bound sieve with a deficit $< \frac34$ for $\Phi_1\cup\Phi_2$, which establishes Theorem 2.

\section{Proof of Theorem 1.1}
This is not very different from the proof of Heath-Brown's result \cite{DRHB}, so we shall be brief.

\begin{lemma}

\end{lemma}
We have
\[
\twolinesum{n-2y < p \le n-y}{v|n-p} 1 \ll \frac{y}{\phi(v)} \log{y/v}
\]
for $y \ge 3v, n \ge 2y$.

\begin{proof}
See \cite[Theorem 3.7]{Halb}.
\end{proof}

\begin{lemma}
When $(n,d)=1$ we have
\[
\twolinesum{u \le E}{d|u, (u,n)=1} \frac{u}{\phi(u)} = CE \frac{\omega(d)}{d} f(n) + O\left((ED)^{\frac12} d(n) \phi(d)^{-1} \right)
\]
where
\[
C= \prod_p \left(1 + \frac{1}{p(p-1)}\right)
\]
and $\omega, f$ are multiplicative functions given by
\[
\omega(p^e) = \frac{p^2}{p^2-p+1}, \quad f(p^e) = \frac{(p-1)^2}{p^2-p+1}\,.
\]
\end{lemma}

\begin{proof}
This is \cite[Lemma 3]{DRHB}.
\end{proof}
\vskip5pt
\noindent
\emph{Proof of Theorem 1.1.}  Let $y = \frac12n^{\theta}$ and $Q=y n^{-\frac12} (\log n)^{-2C}$ where $C$ is as in Theorem 1.2 with $A=4$.  We shall estimate
\[
\#\{(p,u,v): n=p+uv, y \le uv < 2y, Q<u \le 2Q, (u,q)=1 \}
\]
in two different ways.  We have
\[
N \ge \twolinesum{n=k+uv}{\rm{(7.1)}} \lambda(k)
\]
with $\lambda$ as in Theorem 1.2 with $n$ in place of $x$, where (7.1) denotes the set of conditions
\beq\label{7.1}
y \le uv < 2y, \quad Q < u \le 2Q, \quad (u,n)=1.
\endq
It is an easy deduction from Theorem 1.2 that
\[
N \ge \twolinesum{Q < u \le 2Q}{(u,n)=1} \ \twolinesum{n-2y < k \le n-y}{u|n-k} \lambda(k)
\gg  \twolinesum{Q < u \le 2Q}{(u,n)=1} \frac{y}{\phi(u) \log n}\,.
\]
On taking $d=1$ in Lemma 7.2, we have
\begin{align*}
\twolinesum{Q < u \le 2Q}{(u,n)=1} \frac{1}{\phi(u)} &\gg \frac{1}{Q} \twolinesum{Q < u \le 2Q}{(u,n)=1} \frac{u}{\phi(u)}\\
& \gg f(n) + O\left(Q^{- \frac12} d(n)  \right)\\
&\gg f(n),
\end{align*}
since $f(n) \gg \phi(n)/n$ and $Q \gg n^{0.05}$. It follows that
\beq\label{7.2}
N \gg \frac{f(n) y}{\log n}\,.
\endq

We now take $\delta$ to be a small positive constant, and assume there are no solutions of $n=p+ab$ for which $1 \le a, b \le n^{\frac12 - \delta}$ and $ab < 2y$, and estimate $N$ from above.  Using the argument in \cite{DRHB}, we find that
\beq\label{7.3}
N \ll \frac{Q}{\log n} \twolinesum{F \le v < 4F}{(v,nK)=1} \frac{v}{\phi(v)}
\endq
where $F=\frac12 y Q^{-1}$ and
\[
K = \prod_{q_1 \le p < q_2} p, \quad q_1 = 2yn^{\delta-\frac12} Q^{-1}, \quad q_2 = \tfrac12 n^{\frac12 - \delta} Q^{-1}.
\]

Following Heath-Brown's application of the Selberg sieve method in \cite{DRHB}, we obtain the bound
\beq\label{7.4}
\twolinesum{F \le v < 4F}{(v,nK)=1} \frac{v}{\phi(v)}\le \frac{3C F f(n)}{G} + O\left(n^{\frac37}\right)
\endq
where, with 
\[
g(m)  = \prod_{p|m} (p-\omega(p))^{-1},
\]
we have
\[
G= \twolinesum{m < n^{\frac16}}{(m,n)=1, m|K} \mu^2(m) g(m) \ge \twolinesum{q_1 \le p < n^{\frac16}}{p\nmid n} g(p).
\]
Thus
\begin{align*}
G &\ge \sum_{q_1 \le p < n^{\frac16}} \frac{1}{p} - \sum_{p|n} \frac{1}{q_1}\\
& = - \log 6 + \log \log n - \log \log q_1 + O\left((\log q_1)^{-1} + q_1^{-1} \log n  \right).
\end{align*}
Now
\begin{align*}
q_1 &= 2 (\log n)^{2C} n^{\delta},\\
\log q_1 &= \delta \log n + O(\log \log n),\\
\log \log q_1 &= \log \delta + \log \log n + O\left((\log n)^{-\frac12}\right).
\end{align*}
It follows that
\beq\label{7.5}
G \ge - \log(6 \delta) + O\left((\log n)^{-\frac12}\right) \ge - \tfrac12 \log \delta
\endq
if $\delta$ is sufficiently small. We now see from \eqref{7.3}, \eqref{7.4}, \eqref{7.5}, that
\[
N \ll \frac{y f(n)}{\log n \log(1/\delta)},
\]
which contradicts \eqref{7.2} if $\delta$ is chosen to be sufficiently small. This completes the proof of Theorem 1.1.

\medskip\noindent
Roger Baker,\\
Department of Mathematics,\\
Brigham Young University,\\
Provo, UT 84602, USA\\
E-mail: baker@math.byu.edu

\medskip
\noindent
{Glyn Harman,}\\
{Department of Mathematics,}\\
{Royal Holloway, University of London,}\\
{Egham, Surrey TW20 0EX, UK}\\
{E-mail: G.Harman@rhul.ac.uk}

\end{document}